\documentclass[reqno]{amsart}
\usepackage[backend=bibtex, sorting=none, bibstyle=alphabetic, citestyle=alphabetic, sorting=nyt, maxbibnames=99, giveninits=true, isbn=false, url=false, doi=false, eprint=false]{biblatex}  %reference manager
\renewbibmacro{in:}{}
\bibliography{references.bib}
\usepackage{amssymb, calrsfs, graphics, graphicx, enumerate, enumitem, url, xcolor, hyperref, mathrsfs, listings, comment}
\usepackage{tabularx}
\usepackage[ruled, lined, linesnumbered, longend]{algorithm2e}
\usepackage[justification=centering]{caption}
\newtheorem{theorem}{Theorem}[section]
\newtheorem{lemma}[theorem]{Lemma}

\newtheorem{proposition}[theorem]{Proposition}

\newtheorem{corollary}[theorem]{Corollary}
\numberwithin{equation}{section}

\theoremstyle{remark}
\newtheorem*{remark}{Remark}

\lstdefinestyle{CStyle}{
    basicstyle=\footnotesize,
    breakatwhitespace=false,         
    breaklines=true,                 
    captionpos=b,                    
    keepspaces=true,                 
    numbers=left,                    
    numbersep=5pt,                  
    showspaces=false,                
    showstringspaces=false,
    showtabs=false,                  
    tabsize=2,
    language=C
}

\title[A new zero-density estimate and PNT]{A new zero-density estimate for $\zeta(s)$ and the error term in the Prime Number Theorem}
\author[C. Bellotti]{Chiara  Bellotti}
\address{School of Science\\
The University of New South Wales, Canberra, Australia}
\email{c.bellotti@unsw.edu.au}
\keywords{Riemann zeta-function, zero-density estimate, prime number theorem}
\subjclass[2020]{11M06, 11M26, 11N05,  11N56}
\begin{document}
\maketitle
\vspace{-2.5em}
\begin{abstract}
We will provide a new type of zero-density estimate for $\zeta(s)$ when $\sigma$ is sufficiently close to $1$. In particular, we will show that $N(\sigma,T)$ can be bounded by an absolute constant when $\sigma$ is sufficiently close to the left edge of the Korobov--Vinogradov zero-free region. As a consequence, we provide the optimal error term in the prime number theorem of the form 
$$
\psi(x)-x \ll x\exp \left\{-(1-\varepsilon) \omega(x)\right\},\qquad \omega(x):=\min _{t \geq 1}\{\nu(t) \log x+\log t\},
$$
where $\nu(t)=A_0(\log t)^{-2/3}(\log\log t)^{-1/3}$ is a decreasing function such that $\zeta(\sigma+it)\neq 0$ for $\sigma\ge 1-\nu(t)$. Precisely, we will show that we can take $\varepsilon=0$.
\end{abstract}
\section{Introduction}\label{section1}
The aim of this article is to provide a new type of zero-density estimate for the Riemann zeta function $\zeta(s)$ when $\sigma$ is sufficiently close to the $1$-line. More precisely, given $1/2<\sigma<1$, $T>0$, we will prove a new type of upper bound for the quantity$$
N(\sigma, T)=\#\{\rho=\beta+i \gamma: \zeta(\rho)=0,0<\gamma<T,\  \sigma<\beta<1\},
$$
which counts the number of non-trivial zeros of $\zeta$ with real part greater than the fixed value $\sigma$ and imaginary part up to $T$. Zero-density estimates usually take the form $N(\sigma,T)\ll T^{f(\sigma)}(\log T)^{D}$, where $f(\sigma)$ is a function of $\sigma$ which tends to $0$ when $\sigma$ approaches $1$. If the Riemann Hypothesis is true, then we have $N(\sigma,T)=0$ for every $1/2<\sigma<1$. It is worthwhile mentioning that the importance of studying $N(\sigma,T)$ allows one to get improvements throughout a variety of topics in number theory, including the error term in the prime number theorem.

Our main theorem is the following result:
\begin{theorem}\label{theorem1general}
Consider $T$ large and $\sigma\ge 1-\frac{K(T)}{(\log T)^{2/3}(\log\log T)^{1/3}}$, with $\alpha<1$, $K(T)\ll (\log\log T)^{\alpha}$. Then 
the following estimate holds
\begin{align*}
    &N(\sigma,T)\ll\exp(B(\log\log T)^\alpha)K(T)=o(\log T),\end{align*}
    where $B>0$ is an effective computable constant.
\end{theorem}
In particular, we will show that $N(\sigma,T)$ can be bounded by an absolute constant when $\sigma$ is sufficiently close to the edge of the Korobov--Vinogradov zero-free region
\begin{equation}\label{kvzerofree}
     \sigma\ge  1-\frac{A_0}{\log^{2/3}T(\log\log T)^{1/3}}=: 1-\nu(T).
\end{equation} This type of result will lead to improvements on the error term in the prime number theorem.
\begin{theorem}\label{edgezerofree}
Let $A_0$ be the Korobov--Vinogradov zero-free region constant in \eqref{kvzerofree}. Consider $A>A_0$ a fixed positive constant and $T$ large. Then, for every $\sigma$ satisfying $\sigma\ge 1-A(\log T)^{-2/3}(\log\log T)^{-1/3}$, one has $N(\sigma,T)=O(1)$. 
\end{theorem}
Theorem \ref{edgezerofree} can be made more precise. Consider the exponential sum
\begin{equation}\label{Sntkv}
S(N, t):= \max _{N<R \leq 2 N}\left|\sum_{N \leq n \leq R} n^{-i t}\right|,
\end{equation}
and denote $\lambda=\log t/\log N$. 
\begin{theorem}\label{th2}
    Consider $S(N,t)$ defined in \eqref{Sntkv} and assume that, for $1 \leq N \leq t$, $S(N, t) \ll N^{1-1/\left(D \lambda^2\right)}$, where $D>0$ is a positive constant. Let $A>A_0$ be a fixed positive constant, where $A_0$ is the Korobov--Vinogradov zero-free region constant in \eqref{kvzerofree}. Then, for $T$ large, and  $\sigma\ge 1-A(\log T)^{-2/3}(\log\log T)^{-1/3}$, the following estimate holds:
$$N(\sigma,T)\ll e^{cA},\qquad c=6\left(\frac{8D}{3}\right)^{1/3}+12.$$
\end{theorem}
The current best known values \cite{bellotti2023explicit} for $A_0$ in \eqref{kvzerofree} when $T$ is sufficiently large and $D$ in Theorem \ref{th2} are $1/48.0718$ and $132.94357$ respectively. Inserting these values into Theorem \ref{th2} we obtain the following estimate.
\begin{corollary}\label{cor3}
    Consider $A>1/48.0718$ a fixed positive constant and $T$ large. the following estimate holds for every $\sigma\ge 1-A(\log T)^{-2/3}(\log\log T)^{-1/3}$:
    \[
    N(\sigma,T)\ll e^{55A}.
    \]
\end{corollary}
\begin{remark}
    Note that the value $55$ in Corollary \ref{cor3} can be reduced if the constant $D$ appearing in the bound for $S(N,t)$ can be improved.
\end{remark}
It is worthwhile mentioning that Theorem \ref{theorem1general}, Theorem \ref{th2} and Corollary \ref{cor3} can be all made fully explicit. However, the aim of this paper is to focus on the theoretical aspect of this new type of zero-density bound and how we can benefit from Theorem \ref{th2} to improve the error term in the prime number theorem $\psi(x)\sim x$, where $\psi(x)=\sum_{\substack{p^k \leq x}} \log p$, with $p$ is a prime. More precisely, one aims to find sharp bounds for the quantity
\begin{equation}\label{Delta}
    \Delta(x)=\frac{|\psi(x)-x|}{x}.
\end{equation}
Ingham \cite{ingham1932distribution}  showed that if $\zeta(s)$ has no zeroes with real part $\sigma \geq 1-\nu(t)$, where $\nu(t)$ is a decreasing function, then one has
$$
\Delta(x) \ll \exp \left\{-\frac{1}{2}(1-\varepsilon) \omega(x)\right\},\qquad \omega(x):=\min _{t \geq 1}\{\nu(t) \log x+\log t\}.
$$
Pintz \cite{pintz80} showed that Ingham's bound can be substantially improved by replacing the factor $1-\varepsilon$ with $2-\varepsilon$, getting
$$
\Delta(x) \ll \exp \left\{-(1-\varepsilon) \omega(x)\right\}.
$$
If $\nu(t)$ is the classical zero-free region, one can take $\varepsilon=0$ as shown in \cite{johnston2024zerodensityestimatesoptimalityerror}. If instead $\nu(t)=A_0(\log t)^{-2/3}(\log\log t)^ {-1/3}$ is the Korobov--Vinogradov zero-free region, then
\begin{equation}\label{defomega}
    \omega(x)=\left(\frac{5^6A_0^3}{2^2 \cdot 3^4 }\right)^{1 / 5} \frac{(\log x)^{3 / 5}}{(\log \log x)^{1 / 5}}=d \frac{(\log x)^{3 / 5}}{(\log \log x)^{1 / 5}},
\end{equation}
and
\begin{equation}\label{kvnotoptimal}
    \Delta(x) \ll \exp \left\{-(1-\varepsilon)d (\log x)^{3/5}(\log\log x)^{-1/5}\right\}.
\end{equation}
A natural question is whether it is possible to prove the converse of Ingham's result. More precisely, assuming that $\zeta(s)$ has (perhaps infinitely many) zeros in the region $1-\nu (t)$, with $\nu(t)=A(\log t)^{-2/3}(\log\log t)^ {-1/3}$, is it possible to obtain $\varepsilon=0$ in \eqref{kvnotoptimal}? In the affirmative case, we would have the optimal error term in the prime number theorem of the form \eqref{kvnotoptimal}. Partial results of the converse of Ingham's theorem can be found in \cite{stas_uber_1961, pintz80}.
In \cite{johnston2024zerodensityestimatesoptimalityerror}, an explicit expression for $\varepsilon$ in terms of $\omega(x)$ is provided, with $\varepsilon= (\log x)^9(\log \log x)^{-3}$ (see \cite{broucke2025connectionzerofreeregionserror} for similar results regarding Beurling zeta functions). We will improve \cite{johnston2024zerodensityestimatesoptimalityerror} by showing that it is possible to take $\varepsilon=0$, proving the converse of Ingham's result and getting the optimal error term in the prime number theorem of the form \eqref{kvnotoptimal}.
\begin{theorem}\label{th3}
Consider $\Delta$ defined in \eqref{Delta}. The following estimate holds:
    \begin{equation*}
        \Delta(x)\ll \exp(55A_0)\exp(-\omega(x)),
    \end{equation*}
    where $A_0$ is the Korobov--Vinogradov zero-free region constant in \eqref{kvzerofree} and $\omega(x)$ is defined in \eqref{defomega}.
\end{theorem}
Note that Theorem \ref{th2} is equivalent to  \eqref{kvnotoptimal} with $\varepsilon=0$. Theorem \ref{th2} has an immediate corollary. Consider
$$
\pi(x)=\sum_{p \leq x} 1, \qquad \theta(x)=\sum_{p \leq x} \log p,
$$
where $p$ is a prime, and define the error terms
\begin{equation}\label{delta12}
    \Delta_1(x)=\frac{|\pi(x)-\operatorname{li}(x)|}{x / \log x}, \qquad \Delta_2(x)=\frac{|\theta(x)-x|}{x}.
\end{equation}
\begin{corollary}
Given $\Delta_{i=1,2}$, as in \eqref{delta12}, the following estimate holds:
    \begin{equation*}
        \Delta_i\ll \exp(55A_0)\exp(-d(\log x)^{3/5}(\log\log x)^{-1/5}),
    \end{equation*}
    where $A_0$ and $d$ are the same as in Theorem \ref{th3}.
\end{corollary}
\section{Background}\label{lemmas}
We start recalling some useful results.
\subsection{Zeros and bounds for $\zeta$}As mentioned in Section \ref{section1}, the asymptotic Korobov--Vinogradov zero-free region for $\zeta(s)$ in \eqref{kvzerofree} is the widest for $|T|$ sufficiently large and the current best-known value for $A_0$ is $A_0=1/48.0718$ \cite{bellotti2023explicit}.
Another essential element in the proof of Theorem \ref{theorem1general} is the number of zeros of $\zeta$ inside a circle of small radius centered on the $1$-line. We require some preliminary results.
\begin{proposition}[\cite{bellotti2023explicit}, Th.1.1]\label{mybound}
    The following estimate holds for every $|t|\ge 3$ and $\frac{1}{2}\le \sigma\le 1$:
\begin{equation*}
    \begin{aligned}
    |\zeta(\sigma+it)|&\le A |t|^{B
    (1-\sigma)^{3/2}}\log^{2/3}|t|,
    \end{aligned}
\end{equation*}with $A=70.6995$ and $B=4.43795$.
\end{proposition}
\begin{lemma}[\cite{titchmarsh_theory_1986}, $\S 3.9$, Lemma $\alpha$]\label{lemmaalpha}
    If $f(s)$ is regular, and
    \[
    \left|\frac{f(s)}{f(s_0)}\right|<e^M\quad (M>1)
    \]
    in the circle $|s-s_0|\le r$, then 
    \[
    \left|\frac{f'(s)}{f(s)}-\sum_\rho\frac{1}{s-\rho}\right|<\frac{AM}{r}\quad (|s-s_0|\le r/4),
    \]
    where $\rho$ runs through the zeros of $f(s)$ such that $|\rho-s_0|\le r/2$.
\end{lemma}
\begin{lemma}[\cite{titchmarsh_theory_1986}, $\S 3.9$, Lemma $\beta$]\label{lemmabeta}
 If $f(s)$ satisfies the conditions of Lemma \ref{lemmaalpha}, and has no zeros in the right-hand half of the circle $\left|s-s_0\right| \leq r$, then
$$
-\operatorname{Re}\left\{\frac{f^{\prime}\left(s_0\right)}{f\left(s_0\right)}\right\}<\frac{A M}{r},
$$
while if $f(s)$ has a zero $\rho_0$ between $s_0-\frac{1}{2} r$ and $s_0$, then
$$
-\operatorname{Re}\left\{\frac{f^{\prime}\left(s_0\right)}{f\left(s_0\right)}\right\}<\frac{A M}{r}-\frac{1}{s_0-\rho_0}.
$$
\end{lemma}
Now, we can bound the number of zeros in a circle of small radius centered on the $1$-line.
\begin{lemma}[Number of zeros in a disk]\label{numberzerosdisknew}
 If $K \ll \log \log t$ then
$$
\#\left\{\rho:|1+i t-\rho| \leq \frac{K}{(\log t)^{2 / 3}(\log \log t)^{1 / 3}}\right\} \ll K.
$$
\end{lemma}
\begin{proof}
 Apply Lemma \ref{lemmaalpha} with $f(s)=\zeta(s)$ and the following parameters:
 \begin{equation}\label{paramzerocircle}
     s_0=\sigma_0+i t=1+\frac{K}{(\log t)^{2 / 3}(\log \log t)^{1 / 3}}+i t,\qquad r=2\left(\frac{\log \log t}{\log t}\right)^{2 / 3}.
 \end{equation}
By Proposition \ref{mybound} and the relation
\[\left|\frac{1}{\zeta(\sigma+it)}\right|<\frac{1}{\sigma-1},\quad \sigma=1+\delta\]
where $\delta$ is sufficiently small, we have
\begin{align*}
    &\left|\frac{\zeta(\sigma+it)}{\zeta(\sigma_0+it)}\right|\ll\exp\left(4.43795(1-\sigma)^{3/2}\log t\right)(\log t)^{4/3}(\log\log t)^{1/3}K^{-1}\\
    &\ll\exp\left(4.43795K^{3/2}(\log\log t)^{-1/2}\right)(\log t)^{4/3}(\log\log t)^{1/3}K^{-1}\ll \exp(b_1\log\log t).
\end{align*}
Hence, we may take
$$
M=b_1\log\log t
$$
in Lemma \ref{lemmaalpha}, getting 
$$
\left|\frac{\zeta^{\prime}}{\zeta}(s)-\sum_{\left|\rho-s_0\right| \leq r / 2} \frac{1}{s-\rho}\right| \le  \frac{6M}{r} \le 3b_1 (\log \log t)^{1 / 3}(\log t)^{2 / 3},
$$
provided $\left|s-s_0\right| \leq r / 4$. In particular, the previous inequality is true for $s=s_0$. Lemma \ref{lemmabeta} implies that
$$
-\operatorname{Re}\left\{\frac{\zeta^{\prime}}{\zeta}\left(s_0\right)\right\} \le 6b_1(\log \log t)^{1 / 3}(\log t)^{2 / 3}.
$$
It follows that
\begin{equation}\label{countzeros}
    \sum_{\left|\rho-s_0\right| \leq r / 2}\operatorname{Re}\left\{ \frac{1}{s_0-\rho}\right\}=\sum_{\left|\rho-s_0\right| \leq r / 2} \frac{\sigma_0-\beta}{\left|s_0-\rho\right|^2}\le6b_1(\log \log t)^{1 / 3}(\log t)^{2 / 3}.
\end{equation}
Since $\sigma_0-\beta>0$ by our choice of $\sigma_0$ in \eqref{paramzerocircle}, each single term $\frac{\sigma_0-\beta}{|s_0-\rho|^2}$ in the above sum in \eqref{countzeros} is positive for every fixed $\rho$. Hence, using the fact that
$$
\frac{K}{(\log t)^{2 / 3}(\log \log t)^{1 / 3}} \le \frac{r}{2} \quad \text { provided } \quad K \le \log \log t,
$$
the relation \eqref{countzeros} implies
$$
\sum_{\left|\rho-s_0\right| \leq K /\left((\log \log t)^{1 / 3}(\log t)^{2 / 3}\right)} \frac{\sigma_0-\beta}{\left|s_0-\rho\right|^2} \le 6b_1(\log \log t)^{1 / 3}(\log t)^{2 / 3}.
$$
Now, by definition of $s_0$, 
$$
\sigma_0-\beta \ge \frac{K}{(\log t)^{2 / 3}(\log \log t)^{1 / 3}}
$$
and, if $\rho$ satisfies
$$
\left|\rho-s_0\right| \leq\frac{K}{(\log \log t)^{1 / 3}(\log t)^{2 / 3}},
$$
then
$$
\frac{1}{\left|s_0-\rho\right|^2} \ge\left(\frac{K}{(\log t)^{2 / 3}(\log \log t)^{1 / 3}}\right)^{-2}.
$$
Hence, combining with \eqref{countzeros} we get
$$
\frac{1}{K} \#\left\{\rho:|1+i t-\rho| \leq \frac{K}{(\log \log t)^{1 / 3}(\log t)^{2 / 3}}\right\} \le 6b_1,
$$
which implies the theorem.
\end{proof}
\subsection{Weights}
We start recalling the following Barban-Vehov weights $\left(\psi_d\right)_d,\left(\theta_d\right)_d$:
$$
\psi_d=\left\{\begin{array}{ll}
\mu(d) & 1 \leq d \leq U, \\ \\
\mu(d) \frac{\log (V / d)}{\log (V / U)} & U<d \leq V, \\ \\
0 & d>V,
\end{array} \right.
\qquad\quad\theta_d=\left\{\begin{array}{ll}
\mu(d) \frac{\log (W / d)}{\log W} & 1 \leq d \leq W, \\ \\
0 & d>W.
\end{array}\right.$$
Then, we denote
\begin{equation}\label{definepsisquare}
    \Psi(n)=\sum_{d \mid n} \psi_d,\qquad \Theta(n)=\sum_{d \mid n} \theta_d.
\end{equation}
We observe that $\theta_d$ is a special case of $\psi_d$ with $U=1$ and $V=W$. Also, $\Psi(n)=0$ for $2\le n\le U$. Furthermore, for $n\le W$,
\begin{equation}\label{thetaaslambda}
   \Theta(n) = \sum_{e | n}\mu(e)\frac{\log(W/e)}{\log W} = \sum_{e|n}\mu(e)-\frac{1}{\log W}\sum_{e|n}\mu(e)\log e=\sum_{e|n}\mu(e)+\frac{\Lambda(n)}{\log W} ,
\end{equation}
and, in particular, if $ 1<n\le W$, then it is exactly equal to $\frac{\Lambda(n)}{\log W}.$ Using a similar argument to the one for Lemma 3.2 in \cite{BELLOTTIlogfree}, there exist $d_1,d_2>0$ positive constants such that
\begin{equation}\label{constd1}
     \sum_{n=1}^N\Theta^2(n)\le \frac{d_1N}{\log W},\qquad N>W.
\end{equation}
and
\begin{equation}\label{constd2}
     \sum_{n=1}^N\Psi(n)\Theta(n)\le \frac{d_2N}{\log (V/U)},\qquad N>VW.
\end{equation}
\subsection{Exponential sums}As already mentioned in Section \ref{section1}, there are known bounds for the exponential sum $S(N,t)$ defined in \eqref{Sntkv}, when $2\le N\le t$. We recall the following bound which will be used for the proof of Corollary \ref{cor3}:
\begin{lemma}[\cite{bellotti2023explicit}, Th. 1.5]\label{expsumbel}
Suppose $N \geq 2$ is a positive integer, $N \leq t$ and set $\lambda=\frac{\log t}{\log N}$. Then
$$
S(N, t) \leq 8.7979 N^{1-1 /\left(D \lambda^2\right)},\qquad D=132.94357.
$$
\end{lemma}
Another bound for $S(N,t)$ defined in \eqref{Sntkv} follows from an application of Theorem 5.9 in \cite{titchmarsh_theory_1986} with $f(x)=-\frac{t\log x}{2\pi}$.
\begin{lemma}\label{secondderiv}
    Consider $S(N,t)$ as in \eqref{Sntkv} and denote $f(x)=-\frac{t\log x}{2\pi}$. Suppose that, for $N<R\le 2N$, $$
\left|f^{\prime \prime}(x)\right| \asymp \lambda_2, \quad x \in[N, R]
$$
for some $\lambda_2>0$. Then,
$$
S(N,t) \ll N \lambda_2^{1 / 2}+\lambda_2^{-1 / 2}.
$$
\end{lemma}
\subsection{Miscellaneous}
A key tool will be the Hal\'asz--Montgomery inequality.
\begin{lemma} [\cite{jutila77}, Lemma 7]\label{lemma7jutila}
Let $f(s)=\sum_{n\le N} a_n  n^{-s}.
$ Then
$$
\left(\sum_{j=1}^J\left|f\left(s_j\right)\right|\right)^2 \leq \sum_{n=1}^N\left|a_n\right|^2 b_n^{-1} \sum_{j, k=1}^J \bar{\eta}_j \eta_k B\left(\bar{s}_j+s_k\right),
$$
where the $\eta_j$ are certain complex numbers of absolute value $1$ , and
$$
B(s)=\sum_{n=1}^{\infty} b_n n^{-s},
$$
where the $b_n$ are any non-negative numbers such that $b_n>0$ whenever $a_n \neq 0$ and the series $B(s)$ is absolutely convergent for all pairs $s=\bar{s}_j+s_k$.
\end{lemma}
\begin{comment}
    Furthermore, we recall the Lambert function, which is denoted as $W(z)$ for $z \in \mathbb{C}$ and it is defined as the function that satisfies $W(z) e^{W(z)}=z.$
If $z\in \mathbb{R},$ $ W(z)$ can take two possible real values for $-1 / e \leq z<0$. Values satisfying $W(z) \geq-1$ belong to the principal branch $W_0(z)$, while values satisfying $W(z) \leq-1$ belong to the $W_{-1}(z)$ branch. The two branches meet at $z=-1 / e$, where $W_0(-1 / e)=W_{-1}(-1 / e)=-1$. All values of $W(z)$ for $z \geq 0$ belong to the principal branch $W_0(z)$. We will work with the branch $W_{-1}$.
\begin{lemma}[\cite{lambert}, Th. 1]\label{lambertineq}
The Lambert function $W_{-1}\left(-e^{-u-1}\right)$ for $u>0$ is bounded as follows
$$
-1-\sqrt{2 u}-u<W_{-1}\left(-e^{-u-1}\right)<-1-\sqrt{2 u}-\frac{2}{3} u.
$$
\end{lemma}
\end{comment}
Finally, we recall a result on the sum of reciprocals of the least common multiples.
\begin{lemma}\label{reciprocalgcd}The following estimate holds for every $x\ge 1$:
    \[
    \sum_{n,m\le x}\frac{1}{[n,m]}\ll (\log x)^3.
    \]
\end{lemma}
\begin{proof}
    Note that
    \begin{equation*}
\begin{aligned}
 \sum_{n,m\le x}\frac{1}{[n,m]}&=\sum_{n,m\leq x} \frac{(n, m)}{nm} =\sum_{n,m\le x} \frac{\sum_{d \mid n, m} \varphi(d)}{nm} =\sum_{d \leq x} \varphi(d) \sum_{\substack{ n, m \leq x\\ d \mid n, m}} \frac{1}{nm} \\&=\sum_{d \leq x} \frac{\varphi(d)}{d^2} \sum_{1 \leq n^{\prime}, m^{\prime} \leq x / d} \frac{1}{n^{\prime} m^{\prime}}=\sum_{d \leq x} \frac{\varphi(d)}{d^2}\left(\sum_{1 \leq n \leq x / d} \frac{1}{n}\right)^2\\
 &\ll (\log x)^2\sum_{d \leq x}\frac{1}{d}+\log x\sum_{d \leq x}\frac{1}{d}\ll (\log x)^3.
\end{aligned}
\end{equation*}
\end{proof}
\section{Proof of Theorem \ref{theorem1general}}
Consider
\[
R(\sigma_0,T)=\{\sigma+it\ |\ \sigma_0\le\sigma\le 1,\ 0\le t\le T\}
\]
and divide $R(\sigma_0,T)$ into smaller rectangles:
\begin{equation}\label{rectangles}
\sigma_0\le \sigma\le 1,\qquad\qquad n\delta\le t\le (n+1)\delta,\qquad \delta=1-\sigma_0  
\end{equation}
where $n=0,\dots, T/\delta-1$. Furthermore, we observe that $\delta>\nu(T)$, where $\nu$ is defined in \eqref{kvzerofree}.
Then, consider the boxes in which $\zeta(s)$ has at least one zero, and, for each of these selected boxes, choose a representative zero arbitrarily. Considering the even and odd numbers $n$ separately, we get two `` $\delta$-well-spaced" systems. Let $|J|$ denote the cardinality of the system $J$ containing at least half of the selected zeros.\\
Furthermore, we will choose the following quantities
\begin{equation}\label{parameters}
\begin{aligned}
     U&=\exp\left(\frac{c_u(\log\log T)^{\alpha}}{1-\sigma}\right),\quad V=\exp\left(\frac{c_v(\log\log T)^{\alpha}}{1-\sigma}\right),\quad W=\exp\left(\frac{c_w}{1-\sigma}\right),\\X&=\exp\left(\frac{c_x(\log\log T)^{\alpha}}{1-\sigma}\right),\quad Y=\exp\left(\frac{c_y(\log\log T)^{\alpha}}{1-\sigma}\right),
\end{aligned}
\end{equation}
such that $U<V$, $UV<X$, $U<Y<X$ and $W<X$. Our main goal is to estimate 
\begin{equation}\label{principal}
    \left(\sum_{r=1}^J\left|\sum_{n=U}^{XY}\Psi(n)\Theta(n) n^{-\rho}\right|\right)^2.
\end{equation}
Throughout the proof of Theorem \ref{theorem1general}, the variable $c$ (that may be different at each occurrence) indicates an effective computable constant.
\subsection{Estimate for $J$} First of all, we can assume that the quantity $N(\sigma_0,T)$ counts the number of zeros  $\rho=\beta+i\gamma$ with $\sigma_0\le \beta<1$ and imaginary part $X< \gamma\le T$. Indeed,  there are no zeros for $\gamma\le X$, since by the asymptotic zero-free region \eqref{kvzerofree} we have, for $T$ large,
\begin{align*}
    1-\beta&\ge\frac{A_0}{(\log \gamma)^{2/3}(\log\log \gamma)^{1/3}}\ge \frac{c_x^{-2/3}A_0(1-\sigma_0)^{2/3}}{(\log \log T)^{2\alpha/3}(\log(c_x(\log\log T)^{\alpha}/(1-\sigma_0)))^{1/3}}\\&\gg \frac{1
    }{(\log T)^{4/9}(\log \log T)^{2\alpha/3+2/9+1/3}}\gg \frac{K(T)}{(\log T)^{2/3}(\log\log T)^{1/3}}.
\end{align*}
\subsubsection{Lower bound}\begin{lemma}\label{lowerboundlemma}
    The following estimate holds for every representative zero $\rho_r=\beta_r+i\gamma_r$ with $\gamma> X$ and $T$ large:
    \begin{equation*}
    \sum_{r=1}^J\left|\sum_{n=U}^{X}\Psi(n)\Theta(n) n^{-\rho_r}\right|=|J|(1-o(1)).
    \end{equation*}
\end{lemma}
\begin{proof}
The Euler–Maclaurin summation formula for $\zeta(s)$ implies that, for any $s=\sigma+it$ with $X<t<T$ and $\sigma\ge \sigma_0$,
\begin{equation*}
\zeta(s)=\sum_{1 \leq n \leq t} n^{-s}+O\left( \frac{1}{X^{\sigma_0}}\right).
\end{equation*}
Given $Y>VW$, define the Dirichlet polynomial $M(s) $ via the convolution:
\[
M(s) := \sum_{n \le Y} h(n) n^{-s}, \qquad \text{where } h(n) := \sum_{d \mid n} \psi_d \theta_{n/d} = (\psi * \theta)(n).
\]
Hence,
\begin{align*}
  \zeta(s) M(s)=\sum_{n \le XY} \left( \sum_{dm = n} h(m) \right) n^{-s}+M(s)\sum_{X < n \leq t}n^{-s}+O\left(\frac{M(s)}{X^{\sigma_0}}\right),
\end{align*}
and since $h= \psi * \theta $, we have
\[
\sum_{dm = n} h(m) = \sum_{dm = n} (h * 1)(n) = (\psi * \theta * 1)(n) = (\psi * 1)(n) \cdot (\theta * 1)(n) = \Psi(n) \Theta(n).
\]
Therefore, using the fact that $\Psi(n)=0$ for $2\le n\le U$, we have 
\begin{align*}
    \left|M(s)\zeta(s)-1-\sum_{U<n\le XY}\frac{\Psi(n)\Theta (n)}{n^s}\right|\ll \left|M(s)\sum_{X < n \leq t}n^{-s}\right|+\frac{|M(s)|}{X^{\sigma_0}}.
\end{align*}
Evaluating at $s=\rho_r$ gives
\begin{align*}
    \left|-1-\sum_{U<n\le XY}\frac{\Psi(n)\Theta (n)}{n^{\rho_r}}\right|\ll \left|M(\rho_r)\sum_{X < n \leq \gamma_r}n^{-\rho_r}\right|+\frac{|M(\rho_r)|}{X^{\sigma_0}},
\end{align*}
and hence 
\begin{equation*}
    \left|\sum_{U<n\le XY}\frac{\Psi(n)\Theta (n)}{n^{\rho_r}}\right|= 1-o(1),
\end{equation*}
provided that
\begin{equation*}
    \left|M(\rho_r)\sum_{X < n \leq \gamma_r}n^{-\rho_r}\right|+\frac{|M(\rho_r)|}{X^{\sigma_0}}=o(1).
\end{equation*}
By dyadic division, partial summation and the fact that $\beta_r\ge \sigma$, we have
\begin{align*}
     \left|\sum_{X< n\le \gamma_r}n^{-\beta_r-i\gamma_r}\right|&\ll \log(\gamma_r/X)\max_{X\le N\le \gamma_r}\left(N^{-\sigma}\max_{N\le N'<2N}\left|\sum_{N\le n\le N'}n^{-i\gamma_r}\right|\right).
    \end{align*}
Lemma \ref{expsumbel} implies 
\begin{align*}
   \max_{N\le N'<2N}\left|\sum_{N\le n\le N'}n^{-i\gamma_r}\right|\ll  N\exp\left(-\frac{(\log N)^3}{D(\log \gamma_r)^2}\right).
\end{align*}
Hence for every $X\le \gamma_r\le T$ and $\beta_r\ge \sigma$,
\begin{align*}
    \left|\sum_{X\le n\le \gamma_r}n^{-\beta_r-i\gamma_r}\right|&\ll \log T\max_{X\le N\le T}\left(N^{1-\sigma}\exp\left(-\frac{(\log N)^3}{D (\log T)^2}\right)\right)\\
    &= \log T\max_{X\le N\le T}\left(\exp\left((1-\sigma)\log N-\frac{(\log N)^3}{D(\log T)^2}\right)\right).
\end{align*}
The maximum is reached at
\[
\log N=\sqrt{\frac{1-\sigma}{3D}}\log T,
\]
which is always smaller than $\log X$ for our range of $\sigma$. Hence, since the function is decreasing in $N$ when $X\le N\le T$, the maximum will be in $N=X$:
\begin{equation}\label{kvlowerbound}
        \begin{aligned}
    &\left|\sum_{X\le n\le \gamma_r}n^{-\beta_r-i\gamma_r}\right|\ll \log T\left(\exp\left(-\frac{c(\log\log T)^{3\alpha+1}}{D(K(T))^3}\right)\right).
\end{aligned}
\end{equation}
It remains to estimate $|M(\rho_r)|$. Trivially,
\begin{equation}
    |M(\rho_r)|\ll \sum_{n\le Y}\left(\sum_{d|n}1\right)n^{-\beta_r}\ll \sum_{n\le Y}\frac{d(n)}{n^{\sigma}}\ll Y^{1-\sigma}(\log Y)^2,
\end{equation}
where we used the known bound $\sum_{n\le Y}d(n)n^{-1}\ll (\log Y)^2$ and $\beta_r\ge \sigma$.
Lemma \ref{lowerboundlemma} follows.
\end{proof}
\subsubsection{Upper bound} Consider, for a fixed zero $\rho$,
\[
f(\rho)=\sum_{n=U}^{XY}\Psi(n)\Theta(n) n^{-\rho}.
\]
An application of Lemma \ref{lemma7jutila} to $f(\rho)$ with, for $U<n\le XY$,
\begin{align*}
a_n  =\Psi(n) \Theta(n) n^{-1 / 2} \qquad b_n =\Theta^2(n), \qquad s_j=\rho_j-1 / 2,
\end{align*}
and $0$ otherwise, gives
\begin{equation}\label{afterhalaszmont}
 \begin{aligned}
    &   \left(\sum_{r=1}^{|J|}\left|\sum_{n=U}^{XY}\Psi(n)\Theta(n) n^{-\rho_r}\right|\right)^2\le \left(\sum_{n=U}^{XY}\frac{\Psi^2(n)}{n}\right)\left|\sum_{r,s\le |J|}\sum_{n=U}^{XY}\Theta^2(n)n^{1-\beta_r-\beta_k-i(\gamma_r-\gamma_k)}\right|.
\end{aligned}
\end{equation}
We will estimate the two factors separately. For the first factor, we follow \cite{BELLOTTIlogfree}.
\begin{lemma}\label{an2bn}
   The following estimate holds:
    \[
     \sum_{n=U}^{XY}\frac{\Psi^2(n)}{n}=O(1).
    \] 
\end{lemma}
\begin{proof}
  Following \cite{BELLOTTIlogfree} with $z_1=U$, $z_2=V=U^v$, $t=v=c_v/c_u$, one has
    \begin{equation*}
        \begin{aligned}
          &\sum_{n=U}^{XY}\frac{\Psi^2(n)}{n}\le3.09\cdot \frac{\log (XY)}{\log(V/U)}\cdot \frac{1.301(v^2+1)+1.084(v+1)-0.116}{v-1}=O(1).
        \end{aligned}
    \end{equation*}
\end{proof}
For the second factor in \eqref{afterhalaszmont}, we start considering the diagonal terms $r=k$.
\begin{lemma}[Diagonal terms]\label{lem:diagonal}
  The following estimate holds:
 \[\sum_{r\le |J|}\left|\sum_{n=U}^{XY}\Theta^2(n)n^{1-2\beta_r}\right|\ll |J|\exp(2(c_x+c_y)(\log\log T)^{\alpha}). \]
\end{lemma}
\begin{proof}Note that
\begin{align*}
    &\sum_{n=U}^{XY}\Theta^2(n)n^{1-2\beta_r}\le (XY)^{1-2\sigma}\sum_{n\le XY}\Theta^2(n)+(2\sigma-1)\int_{U}^{XY}\left(\sum_{n\le u}\Theta^2(n)\right)\frac{1}{u^{2\sigma}}\text{d}u.
\end{align*}
An application of \eqref{constd1} gives
\begin{align*}
   \int_{U}^{XY}\left(\sum_{n\le u}\Theta^2(n)\right)\frac{1}{u^{2\sigma}}\text{d}u&\le \frac{d_1}{\log W}\int_{U}^{XY}\frac{1}{u^{2\sigma-1}}\text{d}u\le  \frac{d_1(XY)^{2-2\sigma}}{\log W(2\sigma-2)}.
\end{align*}
It follows that
\begin{align*}
    \sum_{n=1}^{XY}\Theta^2(n)n^{1-2\beta_r}&\ll\frac{(XY)^{2-2\sigma}}{(1-\sigma)\log W}\ll  \exp(2(c_x+c_y)(\log\log T)^{\alpha}).
\end{align*}
\end{proof}
Now we deal with the off-diagonal terms. 
\begin{lemma}[Off-diagonal terms]\label{lem:offdiagonal}
    If $r\neq k$ the following estimate holds:
    \begin{align*}
        &\sum_{\substack{r\neq k\\r,k\le |J|}}\left|\sum_{n=U}^{XY}\Theta^2(n)n^{1-\rho_r-\overline{\rho_k}}\right|\ll  |J|(\log|J|)(XY)^{2(1-\sigma)}+o(1)|J|^2.
    \end{align*}
\end{lemma}
\begin{proof} To ease the notation, we denote $t:=\gamma_r-\gamma_k$ and $\alpha :=\beta_r+\beta_k\ge 2\sigma$. We consider two cases, depending on the size of $|t|$.
We start analyzing the case $\delta<|t|<U^{1-\varepsilon}$, where $\varepsilon>0 $ is a sufficiently small fixed constant. By the definition of $\Theta^2(n)$,
\begin{align*}
   &\Sigma:=\sum_{U\le n\le XY}\Theta^2(n)n^{1-\alpha-it}=\sum_{d,e\le W}\theta_d\theta_e\sum_{\substack{U<n<XY\\ [d,e]|n}}\frac{1}{n^{\alpha-1+it}}\\&=\sum_{d,e\le W}\frac{\theta_d\theta_e}{[d,e]^{\alpha-1+it}}\sum_{U/[d,e]<n<XY/[d,e]}\frac{1}{n^{\alpha-1+it}}.
\end{align*}
Applying partial summation to the inner sum
\begin{equation*}
\begin{aligned}
&\sum_{U /[d, e]<n<XY /[d, e]} \frac{1}{n^{\alpha-1+i t}}\\&=\left(\frac{ XY}{[d, e]}\right)^{2-\alpha-i t}-\left(\frac{U}{[d, e]}\right)^{2-\alpha-i t}+(i t+\alpha-1) \int_{U /[d, e]}^{ XY /[d, e]} \frac{u-\{u\}}{u^{\alpha+i t}} \mathrm{~d} u \\
& =\frac{(XY)^{2-\alpha-i t}-U^{2-\alpha-i t}}{[d, e]^{2-\alpha-i t}}+(i t+\alpha-1) \frac{(XY)^{2-\alpha-i t}-U^{2-\alpha-i t}}{[d, e]^{2-\alpha-i t}(2-\alpha-i t)}\\ &\ \ \ -(i t+\alpha-1) \int^{XY/[d, e]}_{ U /[d, e]} \frac{\{u\}}{u^{\alpha+i t}} \mathrm{~d} u \\
& =\frac{(XY)^{2-\alpha-i t}-U^{2-\alpha-i t}}{[d, e]^{2-\alpha-i t}(2-\alpha-i t)}-(i t+\alpha-1) \int^{XY/[d, e]}_{ U /[d, e]} \frac{\{u\}}{u^{\alpha+i t}} \mathrm{~d} u .
\end{aligned}
\end{equation*}
Hence, by Lemma 4.3 of \cite{BELLOTTIlogfree},
\begin{equation}\label{constc4}
    \begin{aligned}
   & \left|\sum_{d,e\le W}\frac{\theta_d\theta_e}{[d,e]^{\alpha-1+it}}\sum_{U/[d,e]<n<XY/[d,e]}\frac{1}{n^{\alpha-1+it}}\right|\\&\le \left((XY)^{2-2\sigma}+U^{2-2\sigma}\right)\left|\frac{1}{(2-\alpha-i t)}\sum_{d,e\le W}\frac{\theta_d\theta_e}{[d,e]} \right|\\&\ \ \  +\left|\sum_{d,e\le W}\frac{(|t|+\alpha-1)}{[d,e]^{\alpha-1}}\frac{[d,e]^{\alpha-1}}{\alpha-1}\left(\frac{1}{U^{\alpha-1}}-\frac{1}{(XY)^{\alpha-1}}\right)\right|\\
    &\ll \frac{(XY)^{2-2\sigma}}{|t|\log W}+\left(\frac{1}{U^{\alpha-1}}-\frac{1}{(XY)^{\alpha-1}}\right)\left|\sum_{d,e\le W}\left(\frac{|t|}{\alpha-1}+1\right)\right|\\
    &\ll \le \frac{(XY)^{2-2\sigma}}{|t|\log W}+\frac{U^{1-\varepsilon}W^2}{U^{2\sigma-1}}+\frac{W^2}{U^{2\sigma-1}}\\
    &\ll\frac{(XY)^{2-2\sigma}}{|t|\log W}+o(1),
\end{aligned}
\end{equation}
provided that $(2\sigma-2+\varepsilon)>0$, which is possible if $\varepsilon>0$ is chosen sufficiently small.
We now focus our attention to the case $|t|\ge U^{1-\varepsilon}$. By dyadic division and partial summation we have
\begin{align*}
    & \left|\Sigma\right|\ll \log(XY/U)\max_{U\le N\le XY}\left(N^{1-\alpha}\max_{N\le N'<2N}\left|\sum_{N\le n\le N'}\Theta^2(n)n^{-it}\right|\right),
\end{align*}
and, similarly to the previous case,
\begin{align*}
   \sum_{N\le n\le N'}\Theta^2(n)n^{-it}=\sum_{d,e\le W}\theta_d\theta_e\sum_{\substack{N<n<N'\\ [d,e]|n}}\frac{1}{n^{it}}=\sum_{d,e\le W}\frac{\theta_d\theta_e}{[d,e]^{it}}\sum_{N/[d,e]<n<N'/[d,e]}\frac{1}{n^{it}}.
\end{align*}
If $|t|\ge N/[d,e]$, then an application of Lemma \ref{expsumbel} and Lemma \ref{reciprocalgcd} gives
    \begin{align*}
        &\left|\sum_{d,e\le W}\frac{\theta_d\theta_e}{[d,e]^{it}}\sum_{N/[d,e]<n<N'/[d,e]}\frac{1}{n^{it}}\right|\le \sum_{d,e\le W}\left|\sum_{N/[d,e]<n<N'/[d,e]}\frac{1}{n^{it}}\right|\\
        &\ll \sum_{d,e\le W} \frac{N}{[d,e]}\exp\left(-\frac{(\log(N/[d,e]))^3}{D\log^2T}\right)\ll  N \exp\left(-\frac{(\log(N/W^2))^3}{D\log^2T}\right) (\log W)^3.
        \end{align*}
    Hence, when $|t|\ge N/[d,e]$,
    \begin{equation}\label{expsumwithvin}
    \begin{aligned}
         \left|\Sigma\right|&\ll  \log(XY/U) (\log W)^3\max_{U\le N\le XY}\left(N^{2-2\sigma} \exp\left(-\frac{(\log(N/W^2))^3}{D\log^2T}\right)\right)\\&\ll  \log(XY/U) (\log W)^3U^{2-2\sigma}\exp\left(-\frac{(\log(U/W^2))^3}{D\log^2T}\right)\\&\ll \frac{(\log\log T)^{\alpha}}{(1-\sigma)^4}\exp\left(-\frac{c(\log\log T)^{3\alpha}}{D(1-\sigma)^3\log^2T}\right).
    \end{aligned}
    \end{equation}
    Since the right-hand side in \eqref{expsumwithvin} is increasing in $1-\sigma$ when $\sigma$ satisfies $\sigma\ge 1-K(T)(\log T)^{-2/3}(\log\log T)^{-1/3}$, one has
  \begin{equation}\label{constc5}
  \begin{aligned}
      &\left|\Sigma\right|\ll\frac{(\log T)^{8/3}(\log\log T)^{4/3+\alpha}}{(K(T))^4}\exp\left(-\frac{c(\log\log T)^{3\alpha+1}}{D(K(T))^3}\right)=o(1).
  \end{aligned}
    \end{equation}
    It remains to consider the case $U^{1-\varepsilon}<|t|<N/[d,e]$. Rewrite 
    \[
    \sum_{N/[d,e]<n<N'/[d,e]}\frac{1}{n^{it}}=\sum_{N/[d,e]<n<N'/[d,e]}e(f(n)),\qquad f(x):=-\frac{t}{2\pi}\log x.
    \]
 Apply Lemma \ref{secondderiv} with
 \[
 \lambda_2:=\frac{|t|[d,e]^2}{(N')^2},
 \]
so that
 \begin{align*}
    \left| \sum_{N/[d,e]<n<N'/[d,e]}\frac{1}{n^{it}}\right|\ll N^{1/2}[d,e]^{1/2}+ \frac{N}{U^{(1-\varepsilon)/2}[d,e]}.
 \end{align*}
 Hence, by Lemma \ref{reciprocalgcd},
 \begin{align*}
      \left|\sum_{d,e\le W}\frac{\theta_d\theta_e}{[d,e]^{it}}\sum_{N/[d,e]<n<N'/[d,e]}\frac{1}{n^{it}}\right|&\ll N^{1/2}W^3+N\frac{(\log W)^3}{U^{(1-\varepsilon)/2}}.
 \end{align*}
 It follows that, for $U^{1-\varepsilon}<|t|<N/[d,e]$,
 \begin{equation}\label{constc6}
    \begin{aligned}
       &\left|\Sigma\right|\ll\frac{(\log\log T)^\alpha }{(1-\sigma)}\max_{U\le N\le XY}\left(N^{3/2-2\sigma}W^3+N\frac{(\log W)^3}{U^{(1-\varepsilon)/2}}\right)\\
        &\ll\frac{(\log\log T)^{\alpha+1/3}(\log T)^{2/3} }{K(T)}\left(\frac{W^3}{U^{-3/2+2\sigma}}+\exp(2(c_x+c_y)(\log\log T)^\alpha)\frac{(\log W)^3}{U^{(1-\varepsilon)/2}}\right)\\
        &=o(1).
    \end{aligned}
    \end{equation}
Combining \eqref{constc4}, \eqref{constc5} and \eqref{constc6}, one has
\begin{equation*}
    \begin{aligned}
         &\sum_{\substack{r\neq k\\r,k\le |J|}}\left|\sum_{n=U}^{XY}\Theta^2(n)n^{1-\rho_r-\overline{\rho_k}}\right|\ll\frac{(XY)^{2(1-\sigma)}}{\log W}\sum_{\substack{r\neq k\\r,k\le |J|}} \frac{1}{|\gamma_r-\gamma_k|}+o(1)|J|^2\\&\ll  |J|(\log|J|)\frac{(XY)^{2(1-\sigma)}}{(1-\sigma)\log W}+o(1)|J|^2\ll  |J|(\log|J|)(XY)^{2(1-\sigma)}+o(1)|J|^2.
    \end{aligned}
\end{equation*}
\end{proof}
Combining Lemma \ref{an2bn}, Lemma \ref{lem:diagonal} and Lemma \ref{lem:offdiagonal} gives
\begin{equation}\label{upperboundfinal}
 \begin{aligned}
    & \left(\sum_{r=1}^J\left|\sum_{n=U}^{XY}\Psi(n)\Theta(n) n^{-\rho}\right|\right)^2\ll|J|(\log|J|)(XY)^{2(1-\sigma)}+o(1)|J|^2.
\end{aligned}
\end{equation}
\subsection{Conclusion}
Combine Lemma \ref{lowerboundlemma} and \eqref{upperboundfinal}:
\begin{align*}
    &(1-o(1))|J|^2\ll |J|\log|J|(XY)^{2(1-\sigma)},
\end{align*}
and hence
\begin{align*}
    |J|\ll  \log|J|(XY)^{2(1-\sigma)}.
\end{align*}
\begin{comment}
    In particular,
\begin{equation}\label{solvewithlamb}
    |J|\ll f(\sigma,T)\log|J|,\qquad \text{with }\quad f(\sigma,T)=\exp(2(c_x+c_y)(\log\log T)^\alpha).
\end{equation}
In order to solve \eqref{solvewithlamb}, define $y=\log|J|$, so that $|J|=e^y$. Hence, \eqref{solvewithlamb} becomes
\[
(f(\sigma,T))^{-1}\le ye^{-y}.
\]
The unique positive solution of the equation  $-ye^{-y}=-(f(\sigma,T))^{-1}$ is
$$
y=-W\left(-\frac{1}{f(\sigma,T)}\right),
$$
where $W$ is the Lambert function. By the properties of $W(z)$, we can conclude that \eqref{solvewithlamb} is true for all $|J|$ such that
\[
|J|\le e^{-W_{-1}(-1/f(\sigma,T))},
\]
where $W_{-1}$ is one of the branches of the $W$ function. Since
\[W_{-1}(-1/f(\sigma,T))=W_{-1}(-e^{\log(1/f(\sigma,T))})=W_{-1}(-e^{-\log(f(\sigma,T))}),\]
an application of Lemma \ref{lambertineq} with $u=\log(f(\sigma,T))-1$ gives
\begin{align*}
    &|J|\le e^{-W(-1/f(\sigma,T))}<e^{1+\sqrt{2u}+u}=f(\sigma,T)e^{\sqrt{2\log(f(\sigma,T))-2}}\ll f(\sigma,T)^{1+o(1)},
\end{align*}
\end{comment}
Since
\[
|J|^{1-o(1)}\ll \frac{|J|}{\log |J|}\ll (XY)^{2(1-\sigma)},
\]
then the following estimate holds:
\begin{equation*}
    |J|\ll \exp(2(c_x+c_y)(1+o(1))(\log\log T)^\alpha).
\end{equation*}
Note that each box counted by $J$ is contained in a circle centered in $1+it$ and radius $r=\sqrt{5}(1-\sigma)/2$. Hence, an application of Lemma \ref{numberzerosdisknew} with 
\begin{equation*}
    K=\frac{\sqrt{5}}{2}(1-\sigma)(\log\log T)^{1/3}(\log T)^{2/3}\ll K(T)
\end{equation*}
implies
\begin{equation*}
\begin{aligned}
   N(\sigma,T)&\ll \exp(2(c_x+c_y)(1+o(1))(\log\log T)^\alpha)K(T).  
\end{aligned}
\end{equation*}
Theorem \ref{theorem1general} follows by taking $B=2(c_x+c_y)(1+o(1)).$
\section{Proof of Theorem \ref{th2}}
Choose $\alpha=0$, and $c_x$ and $c_u$ in \eqref{parameters} sufficiently large such that \eqref{kvlowerbound} and \eqref{constc5} are $o(1)$. More precisely, given $A=K(T)=O(1)$, we choose $c_x$ and $c_u$ such that in \eqref{kvlowerbound}
$$\frac{c_x^3}{DA^3}>1+\varepsilon,$$
while in \eqref{constc5}
$$\frac{(c_u-2c_w)^3}{DA^3}>\frac{8}{3}+\varepsilon,$$
where $\varepsilon>0$ and $D$ is the constant in the bound $S(N,t)\ll N^{1-1/(D\lambda^2)}$ given in the hypothesis.  
Since $c_x>c_v+c_u>2c_u$, $c_x>c_y>c_u$, $c_w<c_u/2$, Theorem \ref{th2} follows by taking
\begin{equation*}
    c_u=\left(\left(\frac{8D}{3}\right)^{1/3}+1\right)A,\quad c_x=2c_u+2A,\quad c_y=c_v=c_u+A,\quad c_w=\frac{9c_u}{10}.
\end{equation*}
\section{Proof of Theorem \ref{th3}}
Denote 
\[
\omega(x)=\min_t\{\nu(t)\log x+\log t\},\quad \nu(t)=\frac{A_0}{(\log t)^{2/3}(\log\log t)^{1/3}},
\]
where $A_0$ is the zero-free region constant in \eqref{kvzerofree}.
Let $x$ be sufficiently large and choose $T=\exp(2\omega(x))$. We have
\begin{equation}\label{truncatedpsi}
    \left|\frac{\psi(x)-x}{x}\right|=\Delta(x) \leq \sum_{|\operatorname{Im}(\rho)| \leq T} \frac{x^{\operatorname{Re}(\rho)-1}}{|\operatorname{Im}(\rho)|}+O\left(\frac{(\log x)^2}{T}\right).
\end{equation}
The error term in \eqref{truncatedpsi} is negligible. Indeed, since $\omega(x) \geq 2 \log \log x$,
$$
\frac{(\log x)^2}{T}=\exp (-\omega(x)) \cdot \frac{(\log x)^2}{\exp (\omega(x))} \ll \exp (-\omega(x)) .
$$
Furthermore, let $H$ be an absolute large positive constant such that the asymptotic Korobov--Vinogradov zero-free region \eqref{kvzerofree} holds. Hence, we may just estimate
$$
\sum_{H<|\operatorname{Im}(\rho)| \leq T} \frac{x^{\operatorname{Re}(\rho)-1}}{|\operatorname{Im}(\rho)|}=\left(\sum_{\substack{H<|\operatorname{Im}(\rho)| \leq T \\ 1 / 2 \leq \operatorname{Re}(\rho) \leq \sigma_0}}+\sum_{\substack{H<|\operatorname{Im}(\rho)| \leq T \\ \sigma_0<\operatorname{Re}(\rho) \leq 1-\nu(\gamma)}}\right) \frac{x^{\operatorname{Re}(\rho)-1}}{|\operatorname{Im}(\rho)|}=\Sigma_1+\Sigma_2,
$$
where
\[
\sigma_0=1-\frac{100A_0}{(\log T)^{2/3}(\log\log T)^{1/3}}=1-\nu_1(T).
\]
Indeed, there are $O(H\log H)$ zeros with imaginary part less than $H$, and this quantity is finite since $H$ is a fixed absolute constant. Regarding $\Sigma_1$, following Pintz's argument \cite{pintz80}, we get
\begin{align*}
    \Sigma_1\ll \exp(-(1-\varepsilon)\omega_2(x)),\qquad \omega_2(x)=\min_t\{100\nu(t)\log x+\log t\}.
\end{align*}
Choosing for instance $\varepsilon=1/100$ one has $\Sigma_1\ll \exp(-\omega(x)).$
To estimate $\Sigma_2$, apply Riemann--Stieltjes integration by parts
\begin{align*}
    \Sigma_2&\ll \int_{H}^T\int_{\sigma_0}^{1-\nu(\gamma)}\frac{x^{\sigma-1}}{\gamma}\text{d}N(\sigma,\gamma)\ll \int_{H}^T\int_{\sigma_0}^{1-\nu(\gamma)} N(\sigma, \gamma) \cdot\left|\frac{\partial^2 }{\partial \sigma \partial \gamma}\left(\frac{x^{\sigma-1}}{\gamma}\right)\right| d \sigma d \gamma.
\end{align*}
Note that for every $\sigma\in [\sigma_0,1-\nu(\gamma))$, we can express $\sigma$ as $$\sigma=1-A(\log T)^{-2/3}(\log\log T)^{-1/3},$$ where $A$ is a positive constant with $100A_0\le A<A_0$. Hence, Corollary \ref{cor3} implies
\begin{align*}
    \Sigma_2&\ll\log x \int_{H}^T\int_{\sigma_0}^{1-\nu(\gamma)} \exp (55 A) \frac{x^{\sigma-1}}{\gamma^2} d \sigma d \gamma\ll \log x\int_{H}^{T}\frac{1}{\gamma^2}\int_{1-\nu_1(T)}^{1-\nu(T)}\exp (55 A)x^{\sigma-1}d\sigma d\gamma.
\end{align*}
Note that
\begin{align*}
    &\int_{1-\nu_1(T)}^{1-\nu(T)}\exp (55 A)x^{\sigma-1}d\sigma\\
    &=-(\log T)^{-2/3}(\log\log T)^{-1/3}\int_{100A_0}^{A_0} \exp(55A)x^{-A(\log T)^{-2/3}(\log\log T)^{-1/3}}dA\\&\ll \exp(55A_0)x^{-A_0(\log T)^{-2/3}(\log\log T)^{-1/3}}(\log x)^{-1}.
\end{align*}
Hence, 
\begin{equation*}
    \Sigma_2\ll \exp(55A_0)\exp(-\omega(x)),
\end{equation*}
with
\[
\omega(x)=\min_t\{\nu(t)\log x+\log t\}=\nu(t_0)\log x+\log t_0,
\]
where $t_0$ minimum. We can conclude that
\begin{equation}
    |\psi(x)-x|\ll x\exp(55A_0)\exp(-\omega(x)).
\end{equation}
\section*{Acknowledgments}
I would like to thank Bryce Kerr, Andrew Yang and my supervisor Timothy S. Trudgian for their support and helpful suggestions throughout the writing of this article.
\printbibliography
\end{document}